%% file: ModularShioda2.tex
\documentclass[reqno]{amsart}

\usepackage{amsmath,amssymb,amsthm,mathrsfs,amsfonts,dsfont,stmaryrd} 
\usepackage{hyperref}
\usepackage{multirow}
\usepackage{amscd}   
\usepackage[all]{xy} 
\usepackage{MnSymbol,graphicx}
\usepackage{mathrsfs,enumitem}
\usepackage[titletoc,toc,title]{appendix}
\usepackage{bbm}

\usepackage[usenames,dvipsnames]{color}

\usepackage{tikz}

\newtheorem{lemma}{Lemma}[section]
\newtheorem{theorem}[lemma]{Theorem}
\newtheorem{proposition}[lemma]{Proposition}

\newtheorem{cor}[lemma]{Corollary}

\newtheorem{claim*}{Claim}

\newtheorem{defn}[lemma]{Definition}

\theoremstyle{definition}
\newtheorem{remark}[lemma]{Remark}

\newtheorem{example}[lemma]{Example}

\newcommand{\A}{{\mathbb A}}

\newcommand{\C}{{\mathbb C}}

\newcommand{\Z}{{\mathbb Z}}

\DeclareMathOperator{\Spec}{Spec}

\def\Lc{\mathcal{L}}
\def\Oc{\mathcal{O}}

\def\thetar{\vartheta}
\newcommand{\car}[2]{%
	\left[{\substack{#1\\#2}}\right]}

\numberwithin{equation}{section}
\numberwithin{table}{section}

\title{On different expressions for invariants of hyperelliptic curves of genus $3$}

 \makeatletter
\let\@wraptoccontribs\wraptoccontribs
\makeatother


\begin{document}

\begin{abstract}

In this paper we give a passage formula between different invariants of genus 3 hyperelliptic curves: in particular between Tsumine and Shioda invariants. This is needed to get modular expressions for Shioda invariants, that is, for example, useful for proving the correctness of numerically computed equations of CM genus $3$ hyperelliptic curves. 

On the other hand, we also get Shioda invariants described in terms of differences of roots of the equation defining the hyperelliptic curve, that has applications for studying the reduction type of the curve. Under certain conditions on its jacobian, we give a criterion for determining the type of bad reduction of a genus $3$ hyperelliptic curve.  
\end{abstract}
\input{authorinfo.tex}

\maketitle

\section{Introduction}\label{Sec:Intro}

Having expressions for the $j$-invariant of an elliptic curve in terms not only of its coefficients but also in terms of theta constants or differences of its roots is useful for constructing elliptic curves with CM by a given order (or with a given number of points over a finite field) or for determining the reduction type of the elliptic curve.
More precisely, let $E/K$ be an elliptic curve over a field of characteristic different from $2$ and $3$ in Short Weierstrass form $$y^2=f(x)=x^3+ax+b=(x-e_1)(x-e_2)(x-e_3).$$

Two such elliptic curves $E$ and $E'$ are isomorphic if and only if they have the same $j$-invariant, i.e., if and only if $j(E)=j(E')$, that can be rewritten as an equality in a weighted projective space $(a:b)=(a':b')\in\mathbb{P}^1_{(2,3)}$.
Recall the following expressions for the $j$-invariant:
$$
j(E)=1728\frac{4a^3}{4a^3+27b^2}=-2^8\cdot3^3\frac{(e_1e_2+e_2e_3+e_3e_1)^3}{(e_1-e_2)^2(e_2-e_3)^2(e_3-e_1)^2}.
$$
It is well-defined because the denominators of these expressions are proportional to the discriminant of the elliptic curve and hence non-zero since an elliptic curve is by definition non-singular:
$$
\Delta(E)=-16(4a^3+27b^2)=16\operatorname{disc}(f)=(e_1-e_2)^2(e_2-e_3)^2(e_3-e_1)^2.
$$
Notice that this condition is also equivalent to the polynomial $f$ having three different roots, that is, $f$ defining a double cover of $\mathbb{P}^1$ ramified in exactly $4$ points.

Assume now that $K$ is a discrete valuation field with ring of integers $\mathcal{O}_K$, uniformizer $\pi$ and valuation $v$, the elliptic curve $E:\,y^2=x^3+ax+b$ with $a,b\in\mathcal{O}_K$ has good reduction modulo $\pi$ if and only if $v(\Delta(E))=0$. More precisely, the following classical result holds:

\begin{theorem}(\cite{Tate75})
	Let $E/K:\,y^2=x^3+ax+b$ be an elliptic curve given by an integer model, i.e., $a,b\in\mathcal{O}_K$. Then:
	\begin{enumerate}
		\item $E$ has good reduction if and only if $v(\Delta(E))=0$.
		\item $E$ has multiplicative (bad) reduction if $v(\Delta(E))>0$ and $v(a)=0$.
		\item $E$ has additive (bad) reduction if $v(\Delta(E))>0$ and $v(a)>0$. After a finite field extension, additive reduction always becomes multiplicative or good.
		\item $E$ has potentially good reduction if and only if $3v(a)\geq v(\Delta(E))$.		
	\end{enumerate}
\end{theorem}

Intuitively, and in terms of the roots of $f$: if the $3$ roots ($4$ if we count the point at the infinity) are different then we have good reduction $(1)$. If $e_1\equiv e_2\equiv e_3\text{ mod }\pi$ we have additive (bad) reduction $(3)$ and after normalizing (maybe over a extension of $K$) we can get at least two roots with different valuation, so we fall in $(1)$ or $(2)$, and we distinguish those cases by the exact valuations of $e_i-e_j$ as in $(4)$. Finally, if $e_1\equiv e_2\nequiv e_3\text{ mod }\pi$ we have multiplicative bad reduction and nothing can be done to improve this, that is, in this case we have geometrically bad reduction.

Assume now that $K=\mathbb{C}$. Then $E$ is also a complex torus $\mathbb{C}/\Lambda$ with $\Lambda=<\tau, 1>$ for some $\tau\in\mathbb{H}$. We can then write the $j$-invariant as:
$$
j(E)=\frac{1}{q}+744+196884q+21493760q^2+864299970q^3+...
$$
where $q=\operatorname{exp}(2\pi i\tau)$, or
$$
j(E)=2^5\frac{(\thetar_1^8+\thetar_2^8+\thetar_3^8)^3}{(\thetar_1\thetar_2\thetar_3)^8},
$$
where $\thetar_1$, $\thetar_2$ and $\thetar_3$ are the three even theta constants of an elliptic curve defined as in \cite[1.3,4.21]{hayquehacer}. We call this last expression a modular expression for the $j$-invariant, since it is given as a (classical) modular function of $\mathbb{H}$ for the modular group $\operatorname{SL}_2(\mathbb{Z})/\{\pm \operatorname{Id}_2 \}$, that is, as a meromorphic function of the modular curve $X(1)$.
This modular expression of the $j$-invariant is for example specially useful for computing Hilbert Class Polynomials of imaginary quadratic orders, that is, elliptic curves with given endomorphism ring. Hence, via de complex multiplication method, elliptic curves over finite fields with a given number of points.

\subsection{Genus $2$ curves}
For genus $2$ curves over a field of characteristic different from $2$, i.e. curves given by an equation $y^2z^4=f(x,z)$ with $\operatorname{deg}(f)=6$ (we write $f(x)=f(x,1)$), we have the Igusa invariants \cite{Igusa60} for binary sextics, defining invariants for genus $2$ curves: $\operatorname{Ig}=(\operatorname{Ig}_2:\operatorname{Ig}_4:\operatorname{Ig}_6:\operatorname{Ig}_{10})$. These invariants determine isomorphism classes: two genus $2$ curves $C,C'$ are isomorphic if and only if $\operatorname{Ig}(C)=\operatorname{Ig}(C')\in\mathbb{P}^{3}_{(2,4,6,10)}$. We write $\operatorname{Ig}_i(C)$ or $\operatorname{Ig}_i(f)$ indifferently when the model of the curve $C$ is fixed and clear from the context.
These invariants can be defined in terms of the differences of roots of $f(x)$ \cite{Clebsch} or as Siegel modular forms for $\mathbb{H}_2$ with respect to the modular group $\operatorname{Sp}_4(\mathbb{Z})$, see \cite{Igusa62, Igusa67}.

The computation and implementation of Class Polynomial for quartic CM fields by using these modular expressions of the invariants are discussed in \cite{Streng14}. Since for genus greater than $1$ the coefficients of Class polynomials are rational numbers and not integers, a bound for their denominators is needed for their exact computation: \cite{BruinierYang,GorenLauter07,yang, LV15b}

The reduction type and the description of the special fiber of the stable model of a genus $2$ curve in terms of the valuation of its invariants is a beautiful result in \cite[Thm. 1]{liug2}.

\subsection{Genus $3$ curves}
Curves of genus $3$ may be non-hyperelliptic or hyperelliptic. In this paper we focus on the hyperelliptic case. For both families we have sets of projective invariants describing the isomorphism classes: Dixmier-Ohno invariants \cite{Dixmier,Ohno} and Shioda invariants \cite{Shioda}. 

In this paper we are interested in giving modular expressions for invariants of hyperelliptic genus $3$ curve. 
This is useful for solving the so-called inverse Jacobian problem: given $\tau\in\mathbb{H}_3$ a period matrix for a hyperelliptic genus $3$ curve $C$ we want to construct an equation of $C$. We can do this numerically as in \cite{BILV}. But if we want to get an exact equation (and prove its correctness) for a curve $C$ let say defined over a number field $K$, we need to have a bound for the denominators appearing in the coefficients of the computed equation, or equivalently for the denominators of a set of invariants. For these, we need a good set of invariants, that is, a set of modular invariants for which we understand the meaning of having a prime dividing the denominator.

In this paper we present such a set of invariants, where the primes in the denominators are only primes of bad reduction which are bounded for the CM case: see \cite{BCLLMNO} and \cite{KLLNOS}. A bound of the exponents for those primes is work in progress in \cite{IKLLMMV_emb}.

Recall that (the modified) Coleman's conjecture \cite[Conj. 4.1]{MoonenOort} states that the number of non-superelliptic CM jacobians of fixed genus $g$ is finite for $g\geq4$, hence the genus $3$ case is important because it is the last case where there are infinitely many generic CM curves. 

Regarding the reduction type of genus $3$ curves, some progress are made in \cite{LLLR} and \cite{BKSW}. In the first of these reference we see that for non-hyperelliptic curves not only the primes of bad reduction need to be bounded for carrying out the CM-method but also the primes of hyperelliptic reduction, which, so far, is not known how to be done. A full generalization of \cite[Thm.1]{liug2} for genus $3$ hyperelliptic curves is a challenging problem. In this paper we study a simplified version of it, see Theorem~\ref{redtype}, by adding some assumptions on the reduction type of the jacobian of the curve. Even if they are strong assumptions, they are for example satisfied by CM curves.

\subsection{Outline} In Section~\ref{Sec:ModInv} we give the precise definition of what we understand by a modular invariant, see Def.~\ref{Def:ModInv}, and we introduce the classical and algebraic theta constants we will use for writing down Siegle modular forms and to determine the type of reduction of certain genus $3$ curves.

We present the different types of genus $3$ hyperelliptic curves invariants in Section~\ref{Sec:genus3Inv}: Shioda's ones and Tsuyumine's ones. We describe the (dis)advantages of each set of generators of invariants of binary octics. Later, in Section~\ref{Sec:relation} we give a \textit{passage formula} from Tsuyumine invariants to Shioda's ones and vice versa, see Theorems~\ref{TinS} and~\ref{SinT}. 

As a consequence of this \textit{passage formula} we provide good (absolute, modular and with some control on its denominators) invariants for performing the CM-method for genus $3$ hyperelliptic curves in Section~\ref{Sec:AbsModInv} and we give a very explicit criterion for determining the reduction type of genus 3 hyperelliptic curves for which the stable reduction of its jacobian is again a principally polarized abelian variety (p.p.a.v.) of dimension $3$ (for example CM curves) in terms of the valuation of its invariants, see Theorem~\ref{redtype} in Section~\ref{Sec:RedType}.

\section{Modular invariants}\label{Sec:ModInv}
Let $A$ be a principally polarized abelian variety of dimension $g$ over the complex numbers. As a complex torus we can write it as $\mathbb{C}^g/(\Z^g+ \tau \Z^g)$ where
$$
\tau\in\mathbb{H}_g:=\{\tau\in\mathbb{C}^{g\times g}:\,\tau^t=\tau,\,\operatorname{Im}(\tau)>0 \}.
$$
A theta function is a holomorphic function of $(z,\tau)\in\mathbb{C}^g\times\mathbb{H}_g$:
$$
\thetar_m(z,\tau)=\thetar\car{m'}{m''}(z,\tau),\,\,m=\car{m'}{m''}\in\mathbb{Z}^{2g}.
$$

Where we define:
\begin{equation}
\thetar(z,\tau) = \sum_{n \in \Z^{g}}\exp(\pi i n \tau n^t + 2 \pi i n z^t), \text{ and }
\end{equation}

\begin{equation*}
\thetar[m](z,\tau) = \operatorname{exp}(\pi im'\tau m'^t+2\pi im'(z+m'')^t)\thetar(z+m''+\tau m'^t,\tau).
\end{equation*}

In this context, $m$ is called a \emph{characteristic} or \emph{theta characteristic}, and the value $\thetar[m](0,\tau)$ is called a \emph{theta constant}.

Given a characteristic $m=\car{m'}{m''}\in\mathbb{Z}^{2g}$, we say that it is even if $m'\cdot m''$ is even. Otherwise we say it is odd. Following \cite[Chapter II, Proposition 3.14]{Mumford1}: 
for $m \in \Z^{2g}$,
\begin{equation*}
\thetar[m](-z,\tau) = (-1)^{m'\cdot m''} \thetar[m](z,\tau).
\end{equation*}
From this we conclude that all odd theta constants vanish. Furthermore, we have that if $n \in 2\mathbb{Z}^{2g}$, then $\thetar[m+n](z,\tau)$ is equal to $\thetar[m](z,\tau)$ or its opposite.
In other words, if $m$ is modified by a vector with even integer entries, the theta value at worst acquires a factor of~$-1$. This is why we say that they are modular forms of weight $1/2$.
More precisely, even theta constants are modular forms of weight $1/2$ with respect to the finite index subgroup $\Gamma_g(4,8)\subset\Gamma_g=\operatorname{Sp}(2g,\mathbb{Z})$. Hence, isomorphic p.p.a.v. have, up to multiple, equal values for $\thetar[m](0,\tau)^2$.

An elliptic curve has $3$ even (and non-zero) theta constants, an indecomposable p.p.a. surface has $10$ even (and non-zero) theta constants and p.p.a. threefolds have $36$ even theta constants that are all non-zero in the non-hyperelliptic jacobian case and exactly one equal to zero in the hyperelliptic case \cite[Lem. 11]{Igusa67}.

We are now going to define theta constants in a more general setting.

\subsection{Algebraic Theta constants}\label{SubSec:algthetas}
Let $R$ be a ring and $S=\Spec R$. Let $X/R$  be an abelian scheme and $\Lc$ be a relatively ample line bundle on
$X$ such that $(-1_X)^* \Lc \simeq \Lc$.  Fix an isomorphism $\epsilon: 0^* \Lc \simeq \Oc_S$ where
$0 : S \to X$ is the zero section. To any $s \in \Gamma(X,\Lc)$, Mumford associates (see \cite[Appendix
I]{Mumford2}) a morphism $\thetar_s : X[2] \to \A^1_S$. 

Following
\cite[Prop.~5.11]{Mumford2} (see also loc. cit. Definition.~5.8), in the special case where $S=\Spec \C$ and
$\Lc$ is the basic line bundle  on $X_{\tau}=\C^g/(\Z^g+ \tau \Z^g)$ (see
loc. cit. p.~36), then $s$ is uniquely defined up to a multiplicative constant  and there is a unique choice of $\epsilon$ such that
\begin{equation} \label{eq:theta}
\thetar_s(x) = e^{-i \pi x_1.x_2} \cdot \thetar \car{x_1}{x_2}(\tau)
\end{equation} 
for any $x \in X[2] \simeq (\frac{1}{2} \Z/\Z)^{2g} \ni (x_1,x_2)$ (after a
specific isomorphism of the $2$-torsion) where $\thetar \car{x_1}{x_2}(\tau)$
is the value at $0$ of the classical theta function with characteristic
$\car{x_1}{x_2}$ \cite[p.192]{Mumford1}.

Let $C/K$ be a genus $3$ curve over a complete discrete valuation field of characteristic different from $2$. Let $(\thetar_i)_{i=1}^{36}$ be its integral algebraic theta constants, i.e., they are integral and the minimum of their valuations is equal to $0$ (this may be attained after multiplying all of the by a common factor). Then:

\begin{enumerate}
	\item $C$ is non-hyperelliptic if and only if all of them are non-zero.
	\item $C$ is a hyperelliptic curve if and only if exactly one of them equal to zero.
	\item $C$ has potentially good non-hyperellitic reduction if and only if all of them have zero valuation (\cite[Thm. 1.6]{LLR}).
	\item $C$ has (pot. good) hyperelliptic reduction if and only if exactly one of them has positive valuation: (1) + same argument than in \cite[Thm. 1.6]{LLR}.
	\item $C$ has geometrically bad reduction if and only if more than one has positive valuation: this is a consequence of (3) and (4).
\end{enumerate}

Statements (1) and (2) are classical, but we refer to the paragraph before the proof of Theorem $1.6$ in \cite{LLR} for a general reference for fields of characteristic different from $2$.

\subsection{Siegel modular forms and curves invariants} Here we follow \cite{LRZ, IKLLMMV}.
Let $X_0^d$ be the non-singular locus of $X_d$: degree $d$ ternary (or binary) forms and let $Y_d$ be the universal curve over the affine space $X_d=\operatorname{Sym}^d(E)$ where $E$ is a vector space of dimension $3$ (or $2$) and $G=\operatorname{GL}(E)$. 

There exists a morphism to the moduli space of genus $g$ curves $p:\,X_0^d\rightarrow M_g$ where $g=(d-1)(d-2)/2$ (or $\lfloor\frac{d-1}{2}\rfloor$). Assume $g>1$. Let $\pi:\,C_{g}\rightarrow M_{g}$ be the universal curve, and let $\lambda$ be the invertible sheaf associated to the Hodge bundle, namely $\lambda=\wedge^g\pi_*\Omega^1_{C_{g}/M_{g}}$. This induces another morphism: $p^*:\Gamma(M_g^0,\lambda^{\otimes h})\rightarrow\Gamma(X_d^0,\alpha^{\otimes h})$: where $\alpha=\wedge^g\pi_*\Omega^1_{Y_d^0/X_d^0}$. Let $\theta:\,M_g\rightarrow A_g$ be the Torelli map. So we have $\theta^*\alpha=\lambda$.

Let $\eta_i$ are a precise basis of $\Omega^1(C_F)=\operatorname{H}^0(C_F,\Omega^1)$ for $F\in X_d^0$ and $\eta=\eta_1\wedge...\wedge\eta_g$. The linear map $\operatorname{Inv}_{gh}(X_d^0)(\text{or }\operatorname{Inv}_{\frac{gh}{2}}(X_d^0))\xrightarrow{\sim}\Gamma(X_d^0,\alpha^{\otimes h}):\,\Phi\mapsto\tau(\Phi)=\Phi\cdot\eta^{\otimes h}$ is an isomorphism.
 Let $\sigma=\tau^{-1}\circ p^*$ be a homomorphism $\Gamma(M_g^0,\lambda^{\otimes h})\rightarrow \operatorname{Inv}_{gh}(X_d^0)$ (or $\operatorname{Inv}_{\frac{gh}{2}}(X_d^0)$) such that $\sigma(f)(F)=f(C_F,(p^*)^{-1}\eta)$. 

Let $F\in X_d^0$ and let $\eta_1,...,\eta_g$ be the basis of regular differential on $C_F$ defined before. Let $\gamma_1,...,\gamma_{2g}$ be a symplectic basis of $\operatorname{H}_1(C,\mathbb{Z})$. The matrix
$$
\Omega=[\Omega_1,\Omega_2]=\begin{pmatrix}
\int_{\gamma_1}\eta_1 & ... & \int_{\gamma_{2g}}\eta_1\\
\vdots& & \vdots\\
\int_{\gamma_1}\eta_g& ... & \int_{\gamma_{2g}}\eta_g
\end{pmatrix}
$$
belongs to the set of Riemann matrices and $\tau=\Omega_2^{-1}\Omega_1\in\mathbb{H}_g$.

\begin{theorem}\label{ThmModular}(\cite[Cor. 3.3.2]{LRZ} and \cite[Cor. 3.6]{IKLLMMV})
	Let $f\in\operatorname{S}_{g,h}(\mathbb{C})$ be a geometric Siegel modular form, $\tilde{f}\in\operatorname{R}_{g,h}(\mathbb{C})$ the corresponding analytic modular form, and $\Phi=\sigma(\theta^*f)$ the corresponding invariant. In the above notation,
	$$
	\Phi(F)=(2i\pi)^{gh}\frac{\tilde{f}}{\operatorname{det}\Omega_2^h}.
	$$ 
\end{theorem}

\begin{defn}\label{Def:ModInv}
	An invariants $\Phi\in\operatorname{Inv}_{gh}(X_d^0)$ (or $\operatorname{Inv}_{\frac{gh}{2}}(X_d^0)$) in the image of $\sigma\circ\theta^*$ is called a modular invariant. In this situation, a modular expression for $\Phi$ is a Siegel modular form $\tilde{f}$ such that the equality in Theorem \ref{ThmModular} holds.
\end{defn}

Modular expressions for Igusa (genus $2$ curves) invariants are in \cite[pag. 848]{Igusa67}. Also in loc. cite, Igusa introduces a map  $\rho$ from Siegel modular forms of genus $g$ to invariants of binary forms of degree $2g+2$. Tsuyumine makes it explicit for $g=3$ in \cite{Tsuyumine2}.

\subsection{Absolute invariants}
As we have already discussed, while we can associated the values of its invariants to a curve, the isomorphism class is determined by those values in a projective space, that is, by these values up to a constant to some power. In order to avoid this ambiguity, we work with what we call absolute invariants.

\begin{defn} A quotient of same degree invariants is called an absolute invariant.
\end{defn}

For example, the isomorphism classes of genus $2$ hyperelliptic curves $C$ with $\operatorname{Ig}_2(C)\neq 0$ are determined by the values of the absolute invariants:
$$
\frac{\operatorname{Ig}_{4}}{\operatorname{Ig}_{2}^2},\frac{\operatorname{Ig}_6}{\operatorname{Ig}_{2}^3},\text{ and }\frac{\operatorname{Ig}_{10}}{\operatorname{Ig}_{2}^5}.
$$

In \cite[Thm. B, p. 631]{Maeda}, Maeda gives $6$ algebraically independent absolute invariants which generate the field of absolute invariants for binary octics. Unfortunately their degrees are too large for practical computations.

\section{Genus 3 hyperelliptic curves invariants}\label{Sec:genus3Inv}

Let $C/K$ be a genus $3$ hyperelliptic curve defined over a field of characteristic different from $2$, then (maybe after a quadratic extension) it is given by an equation $C:\,y^2z^6=f(x,z)$ where $f$ is a degree $8$ homogeneous polynomial. When it is clear from the context and no confusion may appears, we also write $f$ for the univariate polynomial $f(x)=f(x,1)$. 

We refer to \cite{LerRit} for an introduction about hyperelliptic curves (or binary forms) invariants.

\subsection{Shioda Invariants}

In \cite{Shioda} Shioda gives $9$ invariants that generate the algebra of invariants of binary octic forms. The first $6$ invariants are algebraically independent, while the last $3$ are related to the first ones by $5$ explicit relations. The discriminant $D$ of binary octics is an invariant of degree $14$ not contained in this set of invariants given by Shioda but that can be written in terms of them.  Shioda invariants are defined in terms of transvectans: given a binary octic $f$, he defines the covariants
$$
g=(f,f)^4,\, k=(f,f)^6,\, m=(f,k)^4,\,h=(k,k)^2,\,n=(f,h)^4,\,p=(g,k)^4,\,q=(g,h)^4
$$
and the invariants
$$
J_2=(f,f)^8,\,J_3=(f,g)^8,\,J_4=(k,k)^4,\,J_5=k\cdot m,\,J_6=(k,h)^4,$$
$$J_7=m\cdot h,\, J_8=h\cdot p,\,J_9=h\cdot n,\,J_{10}=h\cdot q.
$$
The relations for the last $3$ invariants are of the form:
\begin{eqnarray}
J_8^2+A_6J_{10}+A_7J_9+A_8J_8+A_{16}=0,\label{eq:relations1}\\
J_8J_9+A_7J_{10}+B_8J_9+B_9J_8+B_{17}=0,\label{eq:relations2}\\
J_8J_{10}+\gamma J_9^2+C_8J_{10}+C_9J_9+C_{10}J_8+C_{18}=0,\label{eq:relations3}\\
J_9J_{10}+D_9J_{10}+D_{10}J_9+D_{11}J_8+D_{19}=0,\label{eq:relations4}\\
J_{10}^2+\epsilon J_2J_9^2+E_{10}J_{10}+E_{11}J_9+E_{12}J_8+E_{20}=0.\label{eq:relations5}
\end{eqnarray}
The exact values of the constants are found in \cite[p. 1033]{Shioda}.

The expressions of these invariants in terms of the coefficients of the curve are huge and non-practical for computational purposes. However, the transvectant construction is very efficient. Until the present paper no expression of Shidoa invariants in terms of theta constants or differences of roots of $f$ were known. 

\subsection{Igusa map}\label{sec:IgusaMap}

In \cite{Igusa67} Igusa shows the existence of an exact sequence:
$$
1\rightarrow\chi_{18}A(\Gamma_3)\rightarrow A(\Gamma_3)\xrightarrow{\rho}S(2,8),
$$
where $A(\Gamma_3)$ is the graded ring of modular forms for $\mathbb{H}_3$ with resect to the modular group $\Gamma_3$, $S(2,8)$ is the graded ring of invariants of binary forms of degree $8$ and $$\chi_{18}=\prod_{1}^{36}\thetar_i,$$ where the product runs over the $36$ even theta constants of a genus $3$ curve, is the cusp form of weight $18$ defining the closure of the set of hyperelliptic points \cite{Tsuyumine2}. The invariants in the image of $\rho$ have degree degree divisible by $3$.
The map $\rho$ is unique up to units. Another construction of this map is in \cite{IKLLMMV}.

Tsuyumine proves the following Theorem:

\begin{theorem}(\cite[Chap. 20]{Tsuyumine2})
	The graded ring $A(\Gamma_3)$ of Siegel modular forms of degree $3$ is generated by the following $34$ modular forms: $\alpha_4$, $\alpha_6$, $\alpha_{10}$, $\alpha_{12}$, $\alpha'_{12}$, $\alpha_{16}$, $\alpha_{18}$, $\alpha_{20}$, $\alpha_{24}$, $\alpha_{30}$, $\beta_{14}$, $\beta_{16}$, $\beta_{22}$, $\beta'_{22}$, $\beta_{26}$, $\beta_{28}$, $\beta_{32}$, $\beta_{34}$, $\gamma_{20}$, $\gamma_{24}$, $\gamma_{26}$, $\gamma_{32}$, $\gamma'_{32}$, $\gamma_{36}$, $\gamma_{38}$, $\gamma'_{38}$, $\gamma_{42}$, $\gamma_{44}$, $\delta_{30}$, $\delta_{36}$, $\delta_{46}$, $\delta_{48}$, $\chi_{18}$ and $\chi_{28}$.
\end{theorem}

Moreover, he gives the explicit description of these modular forms in term of the values of the theta constants and the values of $\rho$ evaluated on them in terms of some binary octics invariants given as combinations of differences of roots of a binary octic. 

\subsection{Tsuyumine invariants}

Let $f(x,z)\in K[x,z]$ be a binary octic. Let $\alpha_i,\beta_i\in\bar{K}$ such that $f(x,z)=\prod_{i=1}^{8}(\beta_ix-\alpha_i z)$. We introduce the following notations\footnote{Be careful, they are not exactly the same ones than in \cite{Tsuyumine2}}:
$$
(i_1...i_r)=\prod_{\begin{array}{c}i<j\\i,j\in\{i_1,...,i_r\}\end{array}}(\beta_j\alpha_i-\beta_i\alpha_j),
$$
$$
(i_1...i_r,j_1...j_s)=\prod_{\begin{array}{c}i\in\{i_1,...,i_r\}\\j\in\{j_1,...,j_s\}\end{array}}(\beta_j\alpha_i-\beta_i\alpha_j).
$$
With this notation, the discriminant of $f$ is given by the expression $D=\prod_{i<j}(ij)^2$. 

\begin{proposition}\label{propTsuyu}(\cite[Chap. 5]{Tsuyumine2})
	The graded ring $S(2,8)$ of invariants of binary octics is generated by $I_2$, $I_3$, $I_4$, $I_5$, $I_6$, $I_7$, $I_8$, $I_9$, $I_{10}$, where $I_k$ is an invariant of degree $k$ given as follows:
	\begin{align*}
	I_2&=\sum(12,34)(56,78)
	\\ 
	I_3&=\sum (12)^2(34)^2(56)^2(78)^2(13)(24)(57)(68)\\
	I_4&=\sum (12)^4(345)^2(678)^2\\
	I_5&=\sum (12)^4(345)^2(678)^2(15)(26)(37)(48)\\
	I_6&=\sum (1234)^2(5678)^2\\
	I_7&=\sum (1234)^2(5678)^2(15)(26)(37)(48)\\
	I_8&=\sum (1234)^2(5678)^2(12,56)(34,78)\\
	I_9&=\sum (1234)^2(5678)^2(1,567)(2,678)(3,578)(4,568)\\
	I_{10}&=\sum (1234)^2(5678)^2(15)^2(26)^2(37)^2(48)^2(14,67)(23,58)
	\end{align*}
\end{proposition}

Even if having the same degrees, these invariants are not Shioda invariants \cite{Shioda}. The first seven ones are algebraicaly independent but not the last three, Tsuyumine does not compute the algebraic relations between them.

Tsuyumine invariants enjoys both nice properties we where looking for: a description in terms of diferences of roots of $f$ and a modular expression. However, they are not easily computable, since even if one factors the polynomial $f$ to get its roots, one still needs to go through all the permutation in $S_8$. Their expressions in terms of the coefficients of $f$ are huge. What we do in Proposition~\ref{TinS} is giving them in terms of Shioda invariants who are easily computable. Another nice property about Shioda invariants is that the reconstruction of the curve from them is known and implemented \cite[Sec. 2.3]{LerRit}. We believe both invariants have nice properties, we do not need to decide which ones to use, what we need is a \textit{passage formula} between them.

\begin{remark}
	Tsuyumine invariants are integral invariants, since they are polynomials with integer coefficients on the symmetric functions of the roots of the binary octic, that is, on their coefficients. However, Shioda invariants are not integer invariants, they have some denominators with powers of $2,3,5$ and $7$, see \cite[Section 1.5]{LerRit}. These are precisely the primes for which Shioda invariants are not an HSOP. In his thesis \cite{basson}, gives HSOP's for the characteristics $p=3$ and $7$ and the same techniques may be used to get un HSOP for $p=5$. 
\end{remark}

Following \cite{Tsuyumine2} we have: 

\begin{theorem}(\cite[Sections 23,24,25,26]{Tsuyumine2}) The image of the Siegel modular form $\alpha_4,\,\alpha_6,\,\alpha_{12},\,\beta_{14},\,\beta_{16},\,\beta_{22},\,\gamma_{20},\,\gamma_{24},\,\delta_{30},\,\chi_{18}$ and $\chi_{28}$ by the Igusa $\rho$ map is given by the following combinations of Tsuyumine's invariants:
\begin{align*}
\rho(\alpha_4)&=2^{-3}I_6,
&\rho(\alpha_6)=I_9
\\
\rho(\alpha_{12})&=2^{-3}3^{-2}I_4D,
&\rho(\beta_{14})=I_7D
\end{align*}
\begin{align*}
\rho(\beta_{16})&=I_{10}D,
&\rho(\beta_{22})=I_5D^2\\
\rho(\gamma_{20})&=I_2D^2,
&\rho(\gamma_{24})=I_8D^2\\
\rho(\delta_{30})&=I_3D^3,
&\rho(\chi_{18})=0\\
\rho(\chi_{28})&=D^3.&
\end{align*}
\end{theorem}

These relations can be used to give modular expressions for some invariants of binary octics.

\begin{cor}\label{ModT} Let $C:\,y^2=f(x)$ be a genus $3$ hyperelliptic curve with  period matrix $\Omega=[\Omega_1,\Omega_2]$. Let $\tau=\Omega_2^{-1}\Omega_1$. Then
		$$\begin{array}{rcl}
	I_2(f)D(f)^2=&\rho(\gamma_{20})&=(2i\pi)^{90}\frac{\gamma_{20}(\tau)}{\operatorname{det}\Omega_2^{30}}\\
	I_3(f)D(f)^3=&\rho(\gamma_{30})&=(2i\pi)^{135}\frac{\gamma_{30}(\tau)}{\operatorname{det}\Omega_2^{45}}\\
	I_4(f)D(f)=&2^{3}3^2\rho(\alpha_{12})&=2^{3}3^2(2i\pi)^{54}\frac{\alpha_{12}(\tau)}{\operatorname{det}\Omega_2^{18}}
\\
	I_5(f)D(f)^2=&\rho(\beta_{22})&=(2i\pi)^{99}\frac{\beta_{22}(\tau)}{\operatorname{det}\Omega_2^{33}}\\
	I_6(f)=&2^3\rho(\alpha_{4})&=2^3(2i\pi)^{18}\frac{\alpha_{4}(\tau)}{\operatorname{det}\Omega_2^{6}}\\
	I_7(f)D(f)=&\rho(\beta_{14})&=(2i\pi)^{63}\frac{\beta_{14}(\tau)}{\operatorname{det}\Omega_2^{21}}\\
	I_8(f)D(f)^2=&\rho(\gamma_{24})&=(2i\pi)^{108}\frac{\gamma_{24}(\tau)}{\operatorname{det}\Omega_2^{36}}\\
	I_9(f)=&\rho(\alpha_{6})&=(2i\pi)^{27}\frac{\alpha_{6}(\tau)}{\operatorname{det}\Omega_2^{9}}\\
	I_{10}(f)D(f)=&\rho(\beta_{16})&=(2i\pi)^{72}\frac{\beta_{16}(\tau)}{\operatorname{det}\Omega_2^{24}}\\
	D^3(f)=&\rho(\chi_{28})&=(2i\pi)^{126}\frac{\chi_{28}(\tau)}{\operatorname{det}\Omega_2^{42}}.
	\end{array}$$
	In particular, we have a modular expression for the invariants $I_kD^k$ whose determine the isomorphism class of $C$.
	
\end{cor}

\begin{remark}
	Only invariants of hyperelliptic genus $3$ curves of degree multiple of $3$ are modular. 
\end{remark}

\begin{remark}
	In \cite{cabrones}, the authors give modular expressions for certain combinations of the Dixmier-Ohno invariants of non-hyperelliptic genus $3$ curves (plane quartics).
\end{remark}

\section{Shioda invariants in term of Tsuyumine invariants}\label{Sec:relation}

In this section we compute the relations between Tsuyumine's invariants and Shioda's ones. The computations are just linear algebra and interpolation techniques.

\begin{theorem}\label{TinS}
	Tsuyumine invariants can be written in terms of Shioda invariants as follows:
	\small
	\begin{align*}
	I_2&=2^9\cdot3^2\cdot5\cdot7\cdot J_2, \\
	I_3&=2^{10}\cdot5^2\cdot7^3\cdot J_3, 
	\\
	I_4&=-2^{14}\cdot3\cdot5^2\cdot7\cdot J_2^2 +2^{15}\cdot3^2\cdot7^3 J_4,	\\
	I_5&=2^{14}\cdot3^{-1}\cdot5^3\cdot7^3\cdot J_2J_3 - 2^{14}\cdot3\cdot5\cdot7^4\cdot J_5,
	\\
	I_6&=2^{15}\cdot3^{-1}\cdot7\cdot17^3\cdot J_2^3 - 2^{18}\cdot3\cdot7^3\cdot17\cdot J_2J_4 + 2^{16}\cdot3^2\cdot5\cdot7^4\cdot J_3^2-2^{21}\cdot3\cdot 7^4\cdot J_6,
	\\
	I_7&=2^{16}\cdot5^4\cdot7^3 J_2^2J_3 - 2^{13}\cdot3\cdot5\cdot7^4\cdot17\cdot J_2J_5 - 2^{14}\cdot5\cdot7^5\cdot13\cdot J_3J_4 - 2^{15}\cdot3^2\cdot5\cdot7^5\cdot J_7,
	\end{align*}
	\begin{align*}
	I_8&=2^{15}\cdot3^{-2}\cdot7\cdot17\cdot6469\cdot J_2^4
	-2^{19}\cdot5\cdot7^3\cdot43\cdot J_2^2J_4
	- 2^{16}\cdot3^{-2}\cdot5\cdot7^4\cdot233 \cdot J_2J_3^2\\
	&-  2^{21}\cdot 7^4\cdot 37  \cdot J_2J_6
	+ 2^{18}\cdot 3^2\cdot 5\cdot 7^5 \cdot J_3J_5
	+ 2^{21}\cdot 3\cdot 7^4\cdot J_4^2
	+ 2^{20}\cdot 3^2\cdot 5\cdot 7^5 \cdot J_8,\\
	I_9&=-2^{15}\cdot3^{-3}\cdot 7^3\cdot134489\cdot J_2^3J_3
	+2^{13}\cdot 7^4\cdot 17\cdot 613\cdot J_2^2J_5
	+ 2^{14}\cdot3^{-1}\cdot 5\cdot 7^{5}\cdot1117 \cdot J_2J_3J_4\\
	&	-2^{15}\cdot3\cdot 5^2\cdot 7^5\cdot 19\cdot J_2J_7
	-2^{16}\cdot3\cdot 5\cdot 7^6\cdot J_3^3
	+2^{21}\cdot3^{-1}\cdot5^2\cdot7^6\cdot J_3J_6
	+0\cdot J_4J_5
	-2^{23}\cdot 3^2\cdot 7^6J_9,\\
	I_{10}&=2^{16}\cdot3^{-5}\cdot 7\cdot 17^2\cdot 223\cdot227\cdot J_2^5
	- 2^{21}\cdot3^{-3}\cdot 7^3\cdot 17\cdot 1097 \cdot J_2^3J_4
	+2^{17}\cdot3^{-4}\cdot 5\cdot 7^4\cdot37\cdot991\cdot J_2^2J_3^2\\
	&-2^{22}\cdot3^{-3}\cdot5^2\cdot7^4\cdot421\cdot J_2^2J_6
	- 2^{15}\cdot3\cdot5\cdot7^5\cdot17\cdot29\cdot J_2J_3J_5
	+2^{22}\cdot3^{-1}\cdot7^4\cdot23\cdot31\cdot J_2J_4^2\\
	&+2^{25}\cdot 5\cdot 7^5\cdot17\cdot J_2J_8
	+2^{16}\cdot5\cdot7^6\cdot23 \cdot J_3^2J_4
		-2^{17}\cdot3^2\cdot5\cdot7^6\cdot29\cdot J_3J_7
	+2^{25}\cdot3^{-1}\cdot7^5\cdot61\cdot J_4J_6\\
	&+0\cdot J_5^2
	+ 2^{26}\cdot3\cdot5\cdot7^6 \cdot J_{10}.
	\end{align*}
		
	\normalsize
	
\end{theorem}

\begin{proof}
	Shioda \cite{Shioda} proved that Shioda invariants generate the graded ring $S(2,8)$ of invariants of binary forms. Tsuyumine proved the same thing for his invariants, see Proposition~\ref{propTsuyu}. Hence Tsuyumine invariant $I_n$ has to be a (rational) linear combination of terms $\prod_{i=2}^{10}J_i^{n_i}$ with $\sum_{i=2}^{10}n_i=n$. By taking enough hyperelliptic curves of genus $3$ and computing their Tsuyumine and Shioda invariants we get enough conditions to solve a linear system that produces as solution the coefficients of the combination of Shioda invariants needed to obtain each Tsuyumine invariant. Notice that even if Tsuyumine invariants are a priori difficult to compute, for the interpolation we can consider curves such that $f(x,z)$ completely factors over $\mathbb{Z}$. 
	
	The computations have been performed with Magma \cite{magma}. See Appendix $1$ for the code used for $n=4$.
\end{proof}

\begin{theorem}\label{SinT}
	Shioda invariants can be written in terms of Tsuyumine invariants as follows:
	\small
	\begin{align*}
	J_2=&2^{-9}\cdot3^{-2}\cdot5^{-1}\cdot7^{-1}\cdot I_2,\hspace{10cm} \\
	J_3=&2^{-10}\cdot5^{-2}\cdot7^{-3}\cdot I_3, 
	\\
	J_4= &2^{-19}\cdot3^{-5}\cdot7^{-4} \cdot I_2^2 +2^{-15}\cdot3^{-2}\cdot7^{-3} I_4,	
	\\
	J_5=&2^{-19}\cdot3^{-4}\cdot5^{-1}\cdot7^{-5}\cdot I_2I_3 - 2^{-14}\cdot3^{-1}\cdot5^{-1}\cdot7^{-4}\cdot I_5,\\
	J_6= &-2^{-33}\cdot3^{-8}\cdot5^{-3}\cdot7^{-6}\cdot11\cdot 17 \cdot I_2^3 -2^{-27}\cdot3^{-4}\cdot5^{-1}\cdot7^{-5}\cdot17 \cdot I_2I_4 + 2^{-25}\cdot3\cdot5^{-3}\cdot7^{-6}\cdot I_3^2\\&-2^{-21}\cdot3^{-1}\cdot 7^{-4}\cdot I_6,
	\\
	J_7=&2^{-29}\cdot3^{-7}\cdot5^{-1}\cdot7^{-7}I_2^2I^3 + 2^{-25}\cdot3^{-4} \cdot5^{-2}\cdot7^{-6}\cdot17\cdot I_2I_5 -2^{-26} \cdot3^{-4}\cdot5^{-2}\cdot7^{-6}\cdot13\cdot I_3I_4 -\\ &2^{-15}\cdot3^{-2}\cdot5^{-1}\cdot7^{-5}\cdot I_7,
\\
	J_8=&2^{-39}\cdot3^{-8}\cdot5^{-5}\cdot7^{-9}\cdot13^2\cdot I_2^4 - 2^{-35}\cdot3^{-5}\cdot5^{-3}\cdot7^{8}\cdot59\cdot I_2^2I_4 + 
	2^{-31}\cdot3^{-6}\cdot5^{-5}\cdot7^{-8}\cdot83\cdot I_2I_3^2 \\
	&- 2^{-29}\cdot3^{-5}\cdot5^{-2}\cdot7^{-6}\cdot37\cdot I_2I_6 +
	2^{-26}\cdot3^{-1}\cdot5^{-3}\cdot7^{-7}\cdot I_3I_5 - 2^{-29}\cdot3^{-5}\cdot5^{-1}\cdot7^{-7}\cdot I_4^2 +\\ &2^{-20}\cdot3^{-2}\cdot5^{-1}\cdot7^{-5}\cdot I_8,
\\
	J_9=&2^{-44}\cdot3^{-10}\cdot5^{-5}\cdot7^{-9}\cdot11981\cdot I_2^3I_3 - 2^{-36}\cdot3^{-7}\cdot5^{-3}\cdot7^{-8}\cdot17^2\cdot I_2^2I_5 + \\
	&2^{-38}\cdot3^{-7}\cdot5^{-2}\cdot7^{-8}\cdot31\cdot I_2I_3I_4 + 2^{-32}\cdot3^{-5}\cdot7^{-7}\cdot19\cdot I_2I_7 + 
	2^{-36}\cdot3^{-2}\cdot5^{-5}\cdot7^{-9}\cdot11\cdot I_3^3 -\\ &2^{-3}\cdot3^{-4}\cdot7^{-7}\cdot I_3I_6 + 0\cdot I_4I_5 - 
	2^{-23}\cdot3^{-2}\cdot7^{-6}\cdot I_9,\\
	J_{10}=&-2^{-49}\cdot3^{-11}\cdot5^{-6}\cdot7^{-11}\cdot4177\cdot I_2^5 + 2^{-47}\cdot3^{-8}\cdot5^{-4}\cdot7^{-10}\cdot1433\cdot I_2^3I_4 +\\
	&2^{-41}\cdot3^{-9}\cdot5^{-6}\cdot7^{-11}\cdot7927\cdot I_2^2I_3^2 + 2^{-43}\cdot3^{-9}\cdot5^{-3}\cdot7^{-9}\cdot11^2\cdot1289\cdot I_2^2I_6 -
		\end{align*}
		\begin{align*}
	&2^{-36}\cdot3^{-4}\cdot5^{-4}\cdot7^{-9}\cdot17\cdot I_2I_3I_5 + 2^{-41}\cdot3^{-8}\cdot5^{-2}\cdot7^{-9}\cdot149\cdot I_2I_4^2 -\\ 
	&2^{-30}\cdot3^{-5}\cdot5^{-2}\cdot7^{-7}\cdot17\cdot I_2I_8 - 2^{-39}\cdot3^{-3}\cdot5^{-4}\cdot7^{-10}\cdot59\cdot I_3^2I_4 -\\ &2^{-34}\cdot3^{-1}\cdot5^{-3}\cdot7^{-8}\cdot29\cdot I_3I_7 + 2^{-37}\cdot3^{-5}\cdot5^{-1}\cdot7^{-8}\cdot61\cdot I_4I_6 + 0\cdot I_5^2 +\\ &2^{-26}\cdot3^{-1}\cdot5^{-1}\cdot7^{-6}\cdot I_{10}.
	\\
	\end{align*}
	\normalsize
\end{theorem}

\begin{proof}
	Since none of the coefficients of $J_n$ for $I_n$ in Theorem \ref{TinS} are zero, we can (and do) compute the inverse transformations to obtain Shioda invariants in terms of Tsuyumine invariants.
\end{proof}

\begin{cor}\label{ModS}
	The multiples of Shioda invariants $J_kD^k$ for $k=2,\dots,10$ are modular and a modular expression is given by the formulas in Corollary~\ref{ModT} and Theorem~\ref{SinT}. 
\end{cor}

\begin{cor}
	The relations for the $3$ last Tsuyumine invariants are of the form:
	\begin{eqnarray*}
		I_8^2+A'_6I_{10}+A'_7I_9+A'_8I_8+A'_{16}=0\\
		I_8I_9+B'_7I_{10}+B'_8I_9+B'_9I_8+B'_{17}=0\\
		I_8I_{10}+\gamma' I_9^2+\gamma''I_2I_8^2+C'_8I_{10}+C'_9I_9+C'_{10}I_8+C'_{18}=0\\
		I_9I_{10}+D'_2I_8I_9+D'_9I_{10}+D'_{10}I_9+D'_{11}I_8+D'_{19}=0\\
		I_{10}^2+\epsilon' I_2I_9^2+\epsilon''I_2I_8I_{10}+E'_4I_8^2+E'_{10}I_{10}+E'_{11}I_9+E'_{12}I_8+E'_{20}=0
	\end{eqnarray*}
and the coefficients $A'_i,B'_i,C'_i,D'_i,E'_i,\gamma',\gamma'',\epsilon'$ and $\epsilon''$ can be explicitly computed.	
\end{cor}

\begin{proof}
	Make the substitution of Shioda invariants in equations \ref{eq:relations1},\ref{eq:relations2},\ref{eq:relations3},\ref{eq:relations4} and \ref{eq:relations5} by their expressions in terms of Tsuyumine invariants in Theorem \ref{TinS}.
\end{proof}

\section{Absolute modular invariants for genus $3$ hyperelliptic curves}\label{Sec:AbsModInv}

In order to construct CM genus $3$ hyperelliptic curves by the CM method, we need {modular absolute invariants} whose denominators only contains primes of bad reduction, that is, those ones we are able to control after the results in \cite{BCLLMNO,KLLNOS}. The modular condition is needed to be able to compute them from the period matrix, and the absolute condition is asked for not having to work with the parts $2i\pi$ and $\operatorname{det}\Omega_2$ of the modular expressions of the invariants, see Theorem~\ref{ThmModular}. 

We start with the weighted projective invariants:
$$
(I_2: I_3: I_4: I_5: I_6: I_7: I_8: I_9: I_{10}).
$$

A naive way of proceed will be taking the absolute invariants:
$$
(I^7_2/D, I^{14}_3/D^3, I^7_4/D^2, I^{14}_5/D^5, I^7_6/D^3, I^2_7/D, I^7_8/D^4, I^{14}_9/D^9, I^7_{10}/D^5),
$$
but these ones do not determine isomorphism classes of genus $3$ hyperelliptic curves. For example, after fixing the value of $I_2$ from the one of $D$ and $I^7_2/D$, $I_3$ is only determined up to sign.

The good idea to get the desired properties is to normalize with a degree $1$ quotient of invariants of the form $J=D/J_0$: the degree $1$ conditions permits to fix the isomorphism class, and using $D$ forces to only have primes of bad reduction dividing the numerator of $J$ and hence the denominator of the absolute invariants\footnote{Another advantage of having the discriminant in the denominator of these invariants is that they are always defined since smooth curves have non-zero discriminant.}. We will take $J_0=I_2^2I_3^3$, but other choices would be possible. We obtain the following absolute invariants:
$$
(I_2/J^2, I_3/J^3, I_4/J^4, I_5/J^5, I_6/J^6, I_7/J^7, I_8/J^8, I_9/J^9, I_{10}/J^{10})=
$$
$$
=(\frac{I_2^5I_3^6}{D^2},\frac{I_2^6I_3^{10}}{D^3},\frac{I_2^8I_3^{12}I_4}{D^4},\frac{I_2^{10}I_3^{15}I_5}{D^5}, \frac{I_2^{12}I_3^{18}I_6}{D^6}, \frac{I_2^{14}I_3^{21}I_7}{D^7}, \frac{I_2^{16}I_3^{24}I_8}{D^8}, \frac{I_2^{18}I_3^{27}I_9}{D^9} ,\frac{I_2^{20}I_3^{30}I_{10}}{D^{10}})=
$$
$$
=(\rho(\frac{\gamma_{20}^5\delta_{30}^6}{\chi_{28}^{10}}), \rho(\frac{\gamma_{20}^6\delta_{30}^{10}}{\chi_{28}^{15}}),
\rho(\frac{\gamma_{20}^8\delta_{30}^{12}\alpha_{12}}{\chi_{28}^{19}}),
\rho(\frac{\gamma_{20}^{10}\delta_{30}^{15}\beta_{22}}{\chi_{28}^{24}}),
\rho(\frac{\gamma_{20}^{12}\delta_{30}^{18}\alpha_4}{\chi_{28}^{28}}),
\rho(\frac{\gamma_{20}^{14}\delta_{30}^{21}\beta_{14}}{\chi_{28}^{33}}),
$$
$$
\rho(\frac{\gamma_{20}^{16}\delta_{30}^{24}\gamma_{24}}{\chi_{28}^{38}}),
\rho(\frac{\gamma_{20}^{18}\delta_{30}^{27}\alpha_6}{\chi_{28}^{42}}),
\rho(\frac{\gamma_{20}^{20}\delta_{30}^{30}\beta_{16}}{\chi_{28}^{47}})).
$$

These quotients of modular forms fit into the assumptions of \cite[Thm. 1.1]{IKLLMMV} and look like good candidates to run the CM-method for genus $3$ hyperelliptic curves.

Even if considering very big degree quotients of invariants (this is the price to pay for getting $J=D/J_0$), computations with them are not expected to be very expensive since they involve big powers of just a few modular functions.

\section{Bad reduction of genus $3$ CM hyperelliptic curves}\label{Sec:RedType}
Let $K$ be a field of characteristic different from $2$, and let $A$ be a principally polarized abelian variety of dimension $3$. As we discussed in Section~\ref{SubSec:algthetas} we can attach to it the values of its $36$ even theta constants $\{\thetar_i\}_{i=1}^{36}$.

\begin{theorem}\label{numberofthetas}
	Let $A/K$ be a principally polarized abelian variety of dimension $3$. Then the number of even theta constants equal to $0$ is:
	\begin{enumerate}
		\item[-] $0$ if and only if it is the jacobian of a non-hyperelliptic curve.
		\item[-] $1$ if and only if it is the jacobian of a hyperelliptic curve.
		\item[-] $6$ if and only if it is the product of an elliptic curve and a genus 2 jacobian with the product polarization.
		\item[-] $9$ if and only if it is the product of three elliptic curves with the product polarization. 
	\end{enumerate}
\end{theorem}

\begin{proof} 
	The only possibilities being $0,1,6,9$ is actually proved in \cite{Glass} when the abelian variety is defined over the complex numbers. Since to prove it only the algebraic relations of the theta constants are used and these are still valid for algebraic theta constants, the result remains true for any field.
	
	The only options for a p.p.a.v. of dimension 3 are being a jacobian of a genus 3 curve or the product of an elliptic curve and a p.p.a.v. of dimension $2$. We compute the number of zero theta constants in each case to conclude the result in the Theorem.
	
	The jacobian cases are classic and well-known, see for instance\cite[Section 2]{LLR} for a proof.
	
	The cases with a decomposable polarization are the cases $E\times B$ (with $B$ a genus $2$ jacobian) and $E_1\times E_2\times E_3$ and the argument is analogous in both of them. Let us detail the case $A=E_1\times E_2\times E_3$ with the product polarization. By \cite[Chp.1, Thm. 11]{Rauch}, we have that $\thetar\car{abc}{def}=\thetar\car{a}{d}\thetar\car{b}{e}\thetar\car{c}{f}$.
	Then the $9$ theta constants corresponding to the indices:
	$$
	\car{011}{011},\car{011}{111}, \car{111}{011}, \car{101}{101},\car{101}{111},\car{111}{101},\car{110}{110},\car{110}{111},\car{111}{110}
	$$
	are equal to zero while the other ones are different from zero.
\end{proof}

Let us take now $K$ to be a discrete valuation field of characteristic different from $2$ with uniformizer $\pi$ and valuation $v$. We are interested in studying the reduction type of the stable model of genus $3$ hyperelliptic curves defined over $K$. Let $\{\thetar_i\}_{i=1}^{36}$ be the even theta constants of $C$. There exists an element $\lambda\in\bar{K}^*$ such that the minimal valuation of the normalized values $\{\lambda\cdot\thetar_i\}_{i=1}^{36}$ is equal to $0$. We say that $\{v(\lambda\cdot\thetar_i)\}_{i=1}^{36}$ are the normalized valuations of the theta constants of $C$.

\begin{proposition}\label{propCluster}
	Let $C/K$ be a hyperelliptic genus 3 curve. Then 
	\begin{enumerate}
		\item[(i)] (after normalization) the valuation of exactly $6$ theta constants is positive if and only if there exists a model of $C:\,y^2=f(x)$ for which the associated cluster picture \cite[Def. 1.26]{DDMM2} is 
		 \begin{center}\includegraphics[height=0.5cm]{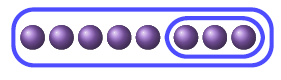}\end{center}
		\item[(ii)] (after normalization) the valuation of exactly $9$ theta constants is positive if and only if there exists a model of $C:\,y^2=f(x)$ for which the associated cluster points structure is 
		 \begin{center}\includegraphics[height=0.5cm]{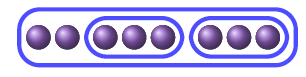}\end{center}
	\end{enumerate} 
Moreover, from those models we can compute the special fiber of the stable model of the curve.
\end{proposition}

\begin{proof}
	We argue as in the proof of \cite[Thm. 1.6]{LLR}. If the cluster pictures are as indicated, then \cite[Table 9.1]{DDMM1} gives the reduction type and Theorem \ref{numberofthetas} gives the number of thetas with positive valuation. Alternatively, for seeing the reduction type, we can also directly compute the positive genus components of the stable model. In the first case:
	$$C:\,y^2=x(x-1)(x-\alpha_3)(x-\alpha_4)(x-\alpha_5)(x-\pi^s\alpha_6)(x-\pi^s\alpha_7),$$
	and the special fiber of the stable model is made of the genus $2$ curve $\bar{y}^2=\bar{x}(\bar{x}-1)(\bar{x}-\bar{\alpha}_3)(\bar{x}-\bar{\alpha}_4)(\bar{x}-\bar{\alpha}_5)$ union the elliptic curve $\bar{y}^2=\bar{x}(\bar{x}-\bar{\alpha_6})(\bar{x}-\bar{\alpha_7})$.
	
	In the second case:
	$$C:\,y^2=x(x-1)(x-\alpha_3)(x-\pi^s\alpha_4)(x-\pi^s\alpha_5)(x-1-\pi^{s'}\alpha_6)(x-1-\pi^{s'}\alpha_7),$$
	and the special fiber of the stable model is the union of the elliptic curves:
	$$
	E_1:\,\bar{y}^2=\bar{x}(\bar{x}-1)(\bar{x}-\bar{\alpha}_3),
	$$
	$$
	E_2:\,\bar{y}^2=\bar{x}(\bar{x}-\bar{\alpha}_4)(\bar{x}-\bar{\alpha}_5),
	$$
	$$
	E_3:\,\bar{y}^2=\bar{x}(\bar{x}-\bar{\alpha}_6)(\bar{x}-\bar{\alpha}_7).
	$$
	
	Now, for the direct implications, as in \cite[Thm. 1.6]{LLR}, this is the most difficult part. We are going to explicitly construct the models. Let us start with a model:
	$$C:\,y^2=x(x-1)(x-\alpha_3)(x-\alpha_4)(x-\alpha_5)(x-\alpha_6)(x-\alpha_7),$$
	where we have taken $\alpha_1=0$, $\alpha_2=1$ and $\alpha_8=\infty$.
	
	Assume that the jacobian of the hyperelliptic curve reduces modulo $\pi$ to the product of three elliptic curves with the product polarization. Proposition \ref{propCluster} implies that (after normalization) exactly $9$ of the even theta characteristic are $0$ modulo $\pi$. Then, as in the proof of Theorem \ref{numberofthetas} we may assume after renaming, that the $9$ theta constants with positive valuations are the ones with characteristic:
	
	$$
	\car{011}{011},\car{011}{111}, \car{111}{011}, \car{101}{101},\car{101}{111},\car{111}{101},\car{110}{110},\car{110}{111},\car{111}{110}.
	$$
	Let us call $v_1, ..., v_9$ the valuations of the theta constants with those theta characteristics.
	
	We consider now Takase formula in \cite[Thm.1.1]{Takase} to write down $a_i=-\frac{a_1-a_i}{a_1-a_2}$, recall that we have taken $a_1=0$ and $a_2=1$. In the same fashion, we will compute $1-a_i=\frac{a_2-a_i}{a_2-a_1}$, $a_3-a_i=\frac{a_3-a_i}{a_3-a_1}$ and so on.
	
	Let us take (following \cite[Def. 1.4.11]{poor}) the map $\eta:\,B\rightarrow\frac{1}{2}\mathbb{Z}^{2g}$ given by 
	\begin{multline*}
	\eta_1=\thetar\car{010}{110},	\eta_2=\thetar\car{001}{110},	\eta_3=\thetar\car{001}{111},	\eta_4=\thetar\car{000}{111},\\	\eta_5=\thetar\car{000}{000},	\eta_6=\thetar\car{100}{000},	\eta_7=\thetar\car{100}{100},	\eta_8=\thetar\car{010}{100}.
	\end{multline*}
This is just a permutation of the indices of the eta map given by Mumford, see \cite{Mumford}.
For this map we have $U_\eta=\{1,3,5,7\}$ and $\vartheta[U_\eta](Z)=\vartheta \car{111}{101} (Z)=0$, corresponding to the only zero even characteristic of a genus $3$ hyperelliptic curve, which is compatible with the previous choices we made. 
In particular, $v_6=\infty$.

Then, by Takase's formula and the previous computations, we have the non-negative valuations:
\begin{multline*}
\hspace{1cm} v(a_1-a_2)=0, \, v(a_1-a_3)=0,\, v(a_1-a_4)=0,\, v(a_1-a_5)=2(v_5-v_4),\,\\ v(a_1-a_6)=2(v_5-v_4),\, 
v(a_1-a_7)=2(v_5-v_4), \, v(a_2-a_3)=2v_3,\, v(a_2-a_4)=2v_3,\,\\ v(a_2-a_5)=2(v_3+v_8-v_4),\,
v(a_2-a_6)=2(v_1+v_8-v_4),\, v(a_2-a_7)=2(v_2+v_8-v_4),\,\\ v(a_3-a_4)=2v_3,\, v(a_3-a_5)=2(v_3+v_7-v_4),\,
v(a_3-a_6)=2(v_1+v_7-v_4),\,\\ v(a_3-a_7)=2(v_2+v_7-v_4),\, v(a_4-a_5)=2(v_3+v_9-v_4),\, \\
v(a_4-a_6)=2(v_1+v_9-v_4),\, v(a_4-a_7)=2(v_2+v_9-v_4),\, 
v(a_5-a_6)=2(v_9+v_4-v_5),\,\\
v(a_5-a_7)=2(v_9+v_4-v_5),\, v(a_6-a_7)=2(v_9+v_4-v_5).\hspace{1.5cm}
\end{multline*}
We detail here the computations for $-a_3=a_1-a_3$: with the notation as in \cite[Thm.1.1]{Takase}, $k=1$, $l=3$, $m=2$, $U_\eta=\{1,3,5,7\}$, $V=\{4,5\}$, $W=\{6,7\}$, so
$$
-a_3=\frac{a_1-a_3}{a_1-a_2}=\left(\frac{\vartheta [47]\vartheta [56]}{\vartheta [2347]\vartheta [2356]}\right)^2=\left(\frac{\vartheta \car{100}{011}\vartheta \car{100}{000}}{\vartheta \car{100}{010}\vartheta \car{100}{001}}\right)^2,
$$
and all these theta constants have zero valuation.

From the previous valuations being non-negative, we get the following inequalities:
\begin{equation*}v_5\geq v_4;\text{  }v_3+v_8, v_1+v_8, v_2+v_8\geq v_4;\text{  }v_3+v_7, v_1+v_7, v_2+v_7\geq v_4;\end{equation*}
\begin{equation*}v_3+v_9, v_1+v_9, v_2+v_9\geq v_4;\text{  }v_9+v_4\geq v_5.\end{equation*}

Now from computing $\frac{a_2-a_i}{a_2-a_j}$ for $i,j\in\{5,6,7\}$ again with Takase's formula and by comparing with the previous valuations, we get $v_1=v_2=v_3$. By computing $\frac{a_6-a_4}{a_6-a_7}$, we get $v_5=v_1+v_7$ and by computing $\frac{a_6-a_2}{a_6-a_5}$, we get $v_5=v_1+v_8$. Hence, $v(a_1-a_5)=v(a_2-a_5)$ and $v(a_1-a_2)=0$ implies $v(a_1-a_5)=v(a_2-a_5)=0$, so $v_4=v_5$. Finally, by computing  $\frac{a_2-a_4}{a_5-a_7}=\frac{a_2-a_4}{a_5-a_4}\frac{a_5-a_4}{a_5-a_7}$, we get $v_7=v_1$.

If we make $v_1=r$, we have $v_1=v_2=v_3=v_7=v_8=r$, $v_4=v_5=2r$ and $v_6=\infty$. Hence, 
\begin{multline*}
\hspace{0.3cm}v(a_1-a_2)=0, \, v(a_1-a_3)=0,\, v(a_1-a_4)=0,\, v(a_1-a_5)=0,\, v(a_1-a_6)=0,\,\\
v(a_1-a_7)=0, \, v(a_2-a_3)=2r,\, v(a_2-a_4)=2r,\, v(a_2-a_5)=0,\,
v(a_2-a_6)=0,\,\\ v(a_2-a_7)=0,\, v(a_3-a_4)=2r,\, v(a_3-a_5)=0,\,
v(a_3-a_6)=0,\, v(a_3-a_7)=0,\,\\ v(a_4-a_5)=0,\, 
v(a_4-a_6)=0,\, v(a_4-a_7)=0,\, 
v(a_5-a_6)=2r,\,
v(a_5-a_7)=2r,\, v(a_6-a_7)=2r.
\end{multline*}

Which yields the cluster picture claimed in the Proposition. For the case of a product of an elliptic curve and a genus $2$ jacobian we proceed in a similar fashion.


\end{proof}

\subsection{A degree 20 invariant}
Let us define the following binary octic invariant of degree $20$:
$$
I_{20}:=\sum_{S_8}\frac{\prod_{i<j} (ij)^4}{(123)^4(45678)^2}=\sum_{S_8}(45678)^2(123,45678)^4
$$
This is indeed an invariant.
\begin{lemma}Let $n$ be an even integer. An expression $I=\sum_{S_n}\prod_{i\neq j} (ij)^{e(i,j)}$ with $e(i,j)\in\mathbb{Z}$ satisfying $\sum_{i}e(i_0,i)+\sum_{i}e(i,i_0)= e\in\mathbb{Z}$ for all $i_0$ defines an invariant of degree $e$ for binary forms of degree $n$. 
\end{lemma}

\begin{proof} Let $f(x,z)=\prod_{i=1}^n(\beta_i x-\alpha_iz)$ be a binary form of degree $n$, then $(ij)=\alpha_i\beta_j-\alpha_j\beta_i$. The condition of $\sum_{i}e(i_0,i)+\sum_{i}e(i,i_0)= e$ being constant is for having $I$ well defined independently of the choice of $\alpha_i$ and $\beta_i$ that can be done up to a multiple. 
	The action of $$M=\begin{pmatrix}
	a & b\\ c & d
	\end{pmatrix}\in\operatorname{GL}_2(\bar{K})$$ on $f$, is by $M(f)=f(M^{-1}(x,z))=\prod_{i=1}^n((d\beta_i+c\alpha_i)x-(b\beta_i+a\alpha_i )z)$, see \cite[Sec. 1]{LerRit}. Hence, the action $M\in\operatorname{GL}_2(\bar{K})$ on $(ij)$ is by sending it to $\operatorname{det}(M)^{-1}\cdot(ij)$ so $M(I(f))=\operatorname{det}(M)^{-\frac{ne}{2}}\cdot I(f)$ and $I$ defines a degree $n$ binary form invariant of weight $\frac{ne}{2}$ and hence of degree $e$.
\end{proof}

\begin{remark}
	Notice that if $e(i_0,j_0)+e(j_0,i_0)$ is odd for some $i_0,j_0\in\{1,\dots,n\}$ then the invariant $I$ is zero since for each term in the sum corresponding to $\sigma\in S_n$ we always have the opposite one corresponding to $\pi_{i_0,j_0}\circ\sigma\in S_n$ where $\pi_{i_0,j_0}$ is the transposition of indexes $i_0$ and $j_0$.
\end{remark}

Given a binary form $f=\beta\prod_{i=1}^8(x-\alpha_iz)$, if three of its roots are equal (let us say $\alpha_1=\alpha_2=\alpha_3$), then only $3!\cdot5!$ terms of $I_{20}(f)$ are not necessarily equal to zero: those that correspond to the terms $(45678)^2(123,45678)^4$ permuted by permutations leaving invariants the sets $\{1,2,3\}$ and $\{4,5,6,7,8\}$. Moreover, all of them are equal and equal to zero if and only if there exists another root equal to $\alpha_1$ or another pair of multiple roots.

Let $K$ be a discrete valuation field with valuation $v$ and characteristic of its residue field equal to $p\neq2$. Let $Sh=\{J_2,\dots,J_{10}\}$ be the set of Shioda invariants if $p\neq3,5,7$ or the corresponding set of invariants forming an HSOP for genus $3$ hyperelliptic curves if $p=3,5,7$, see the paragraph before Corollary $3.18$ in \cite{LLLR} for more details. We define the normalized valuation $v_{Sh}(I(C))$ for any other invariant $I$ of weight $w$ of a genus $3$ hyperelliptic curve $C:\,y^2=f(x)$ as in \cite[pp. 3-4]{LLLR}: $v_{Sh}(I(C))=v(I(C))/w-\operatorname{min}_{J_i\in Sh}\{v(J_{i}(C))/i\}$.


\begin{theorem}\label{redtype}
	Let $C:\,y^2=f(x)$ be a hyeprelliptic genus $3$ curve defined over a discrete valuation ring $\mathcal{O}_K$ whose residue field $k$ has characteristic different from $2$ and such that the reduction of the jacobian of its stable model is still a p.p.a.v. of dimension $3$. Then,
	\begin{itemize}
		\item[(i)] $C$ has potentially good reduction if and only if $v_{Sh}(D)=0$.
		\item[(ii)] $C$ has geometrically bad reduction and the special fiber of its stable model over $\bar{k}$ is isomorphic to the product of one elliptic curve and a genus $2$ jacobian if and only if $v_{Sh}(D)>0$ and $v_{Sh}(I_{20})=0$.
		\item[(iii)] $C$ has geometrically bad reduction and the special fiber of its stable model over $\bar{k}$ is isomorphic to the product of $3$ elliptic curves if and only if $v_{Sh}(D)>0$ and $v_{Sh}(I_{20})>0$.
	\end{itemize}
\end{theorem}

\begin{proof}
	The condition on the reduction of the jacobian is to use Theorem~\ref{numberofthetas} and Proposition~\ref{propCluster}. Part $(i)$ is \cite[Cor. 3.18]{LLLR}. If $v_{Sh}(D)>0$, Theorem~\ref{numberofthetas} tells us that there are only two cases to distinguish: having $6$ or $9$ theta constants with positive valuation. By Proposition~\ref{propCluster} we know that these cases are determined by having models with the cluster picture described there. Then previous discussion on the annulment of $I_{20}$ gives cases $(ii)$ and $(iii)$.
\end{proof}

Now, and again with the notation in \cite{LLLR}:

\begin{cor}\label{CorRed}
	Let $C$ be a hyeprelliptic genus $3$ curve with CM and let $\mathfrak{p}$ be a prime of bad reduction (i.e. $v_{Sh}(D)>0$) with $\mathfrak{p}\nmid 2$. Then the stable reduction of $C$ is made up of three elliptic curves if and only if $v_{Sh}(I_{20}(C))>0$. Otherwise, the stable reduction is the product of an elliptic curve and a principally polarized abelian surface.
\end{cor}

\begin{proof}
	By \cite[Prop. 4.2]{BCLLMNO} the stable reduction of a CM curve is tree-like and we can apply Theorem ~\ref{redtype}.
\end{proof}

%

\begin{lemma}\label{Lem:I20}
	The invariant $I_{20}$ can be written in terms of Shioda invariants as described in the Appendix $2$.
\end{lemma}

\begin{proof}
	We follow the same procedure than in the proof of Theorem \ref{TinS}. But this time the linear system we have to solve has size $102\times102$.
\end{proof}

\begin{example}
	Let us take the genus $3$ hyperelliptic curve $C:\,y^2=x^7+1786x^5+44441x^3+278179x$, it has CM by $\mathbb{Q}[x]/(x^6+13x^4+50x^2+49)$, see \cite[Section 6 - 3rd ex.]{Weng}. Its Shioda invariants are
	\begin{multline*}
		(J_2(C),J_3(C),J_4(C),J_5(C),J_6(C),J_7(C),J_8(C),J_9(C),J_{10}(C))=\\
		(-2^{-2}\cdot3^{2}\cdot7^{-1}\cdot19\cdot475549:0:2^{-9}\cdot7^{-4}\cdot19^2\cdot233\cdot23374404412631:0:\\2^{-14}\cdot3^2\cdot7^{-6}\cdot19^329\cdot1873\cdot12511\cdot4606367\cdot8109203:0:\\
		-2^{-17}\cdot3^{4}\cdot5^{-1}\cdot7^{-9}\cdot11\cdot19^{4}\cdot43\cdot47\cdot2381\cdot4583\cdot11903471\cdot171351716957:0:\\
		-2^{-22}\cdot3\cdot5^{-1}\cdot7^{-11}\cdot19^5\cdot23\cdot50178763\cdot170651941\cdot2743491486709463245193),
	\end{multline*}
	from where $D(C)=2^{18}\cdot7^{24}\cdot11^{12}\cdot19^7$ and we have geometrically bad reduction at $11$ since for example $11\nmid J_{2}(C)$. Because\footnote{Notice we are using Lemma~\ref{Lem:I20} to compute $I_{20}$ since obtaining its roots and using the definition for computing it would be very expensive.}$$I_{20}(C)=2^8\cdot3^3\cdot5^3\cdot7^{12}\cdot19^{10}\cdot131\cdot11867\cdot33730341419\cdot471894282846669530888306233351$$ and $11\nmid I_{20}(C)$, Corollary \ref{CorRed} implies that its stable reduction is a genus $2$ curve union an elliptic curve. Actually, for this particular model of the curve, the reduction modulo $11$ is $y^2=x^3(x-2)(x+2)(x-5)(x+5)$, which has as cluster picture the one described in Proposition~\ref{propCluster} (i). The genus $2$ curve in the special fiber of the stable model is $C_2:\,y^2=x(x-2)(x+2)(x-5)(x+5)$ and, after a suitable change of variable as described in the proof of Proposition~\ref{propCluster}, we also get the elliptic curve: $C_1:\,y^2=x^3+x$. So, $\bar{C}=C_1\cup C_2$ and $\bar{J(C)}=C_1\times J(C_2)$.
	
	We look now at $p=7$. In \cite[Prop. 4.3.3]{basson} the author gives an HSOP for hyperelliptic genus $3$ curves with weights $\{3,4,5,6,10,14\}$. We compute:
$$
	(\mathcal{J}_3(C),\mathcal{J}_4(C),\mathcal{J}_5(C),\mathcal{J}_6(C),\mathcal{J}_{10}(C),\mathcal{J}_{14}(C))\equiv(0, 3,0,2,1,3)\text{ mod }7,
$$
so $p=7$ is a prime of geometrically bad reduction and $v_{Sh}(I_{20}(C))>0$. Hence,  Corollary \ref{CorRed} implies that the special fiber of the stable model is the union of $3$ elliptic curves. Actually, by studying the reduction of $f(x)$ modulo $7$ and arguing as in Proposition~\ref{propCluster} we find that the $3$ elliptic curves have $j$-invariant equal to $1728\equiv6\text{ mod } 7$. 

\end{example}

\section*{Appendix 1}

Here we present as an example the Magma code used for computing Tsuyumine invariant $I_4$ is terms of Shioda invariants in Theorem \ref{TinS}.

\vspace{1cm}

\texttt{S8 := Sym(8);}

\texttt{a:=[1,2,3,4,5,6,7,8];}

\texttt{n:=2;}

\texttt{M:=ZeroMatrix(Rationals(),n,n);}

\texttt{v:=ZeroMatrix(Rationals(),1,n);}

\texttt{R<x>:=PolynomialRing(Rationals());}

\texttt{for i in [1..n] do}

\hspace{1cm}\texttt{a[8]:= 7+i;}

%

\hspace{1cm}\texttt{L:=[(a[1\^}\texttt{g]-a[2\^}\texttt{g])\^}\texttt{4*(a[3\^}\texttt{g]-a[4\^}\texttt{g])\^}\texttt{2*(a[4\^}\texttt{g]-a[5\^}\texttt{g])\^}\texttt{2*}

\hspace{1cm}\texttt{(a[5\^}\texttt{g]-a[3\^}\texttt{g])\^}\texttt{2*(a[6\^}\texttt{g]-a[7\^}\texttt{g])\^}\texttt{2*(a[7\^}\texttt{g]-a[8\^}\texttt{g])\^}\texttt{2*}

\hspace{1cm}\texttt{(a[8\^}\texttt{g]-a[6\^}\texttt{g])\^}\texttt{2: g in S8];}

\hspace{1cm}\texttt{I4:=0;}

\hspace{1cm}\texttt{for j in [1..40320] do}

\hspace{1cm}\hspace{1cm}\texttt{I4:=I4+L[j];}

\hspace{1cm}\texttt{end for;}

\hspace{1cm}\texttt{v[1,i]:=I4;}

\hspace{1cm}\texttt{f:=(x-a[1])*(x-a[2])*(x-a[3])*(x-a[4])*(x-a[5])*(x-a[6])*}

\hspace{1cm}\texttt{(x-a[7])*(x-a[8]);}

\hspace{1cm}\texttt{J2:=ShiodaInvariants(f)[1];}

\hspace{1cm}\texttt{J4:=ShiodaInvariants(f)[3];}

\hspace{1cm}\texttt{C:=Matrix(Rationals(),n,1,[J2\^}\texttt{2,J4]);}

\hspace{1cm}\texttt{M:=InsertBlock(M,C,1,i);}

\texttt{end for;}

\texttt{// [a,b,c,d]*[[J2*J3(1), J2*J3(2)],[J5(1),J5(2)]]=[I5(1),I5(2)]}

\texttt{// x*M=v}

\texttt{Solution(M,v);}

\texttt{> [ -8601600 101154816]}

\texttt{Factorization(8601600);}

\texttt{Factorization(101154816);}

\texttt{>[ <2, 14>, <3, 1>, <5, 2>, <7, 1> ]}

\texttt{>[ <2, 15>, <3, 2>, <7, 3> ]}

\section{Apendix 2}

\begin{multline*}
I20:=-150779522189317809337532416/996451875\cdot J_2^{10} +
291770089409212849667964928/22143375\cdot J_2^8\cdot J_4 -\\
1695148271975113591309205504/66430125\cdot J_2^7\cdot J_3^2 -
179316648534237315000696832/492075\cdot J_2^7\cdot J_6 +
\\
	652065191519301124910743552/2460375\cdot J_2^6\cdot J_3\cdot J_5 +
	16154791364442376858763264/2460375\cdot J_2^6\cdot J_4^2 +\\
109442295325339074758180864/32805\cdot J_2^6\cdot J_8 -
257864348790628803521019904/492075\cdot J_2^5\cdot J_3^2\cdot J_4 +\\
2325823006493876640808960/6561\cdot J_2^5\cdot J_3\cdot J_7 +
323685458436375462012780544/19683\cdot J_2^5\cdot J_4\cdot J_6 -\\
36876419546721924702797824/91125\cdot J_2^5\cdot J_5^2 -
21490047395263280205266944/729\cdot J_2^5\cdot J_{10} -
\\
2273880863076868562712788992/4428675\cdot J_2^4\cdot J_3^4 +
548968706200708073444605952/295245\cdot J_2^4\cdot J_3^2\cdot J_6 -\\
415600353699987918363295744/54675\cdot J_2^4\cdot J_3\cdot J_4\cdot J_5 +
48744958929083435954733056/729\cdot J_2^4\cdot J_3\cdot J_9 -\\
1871598475608430889796632576/91125\cdot J_2^4\cdot J_4^3 -
1276593657056581174310207488/3645\cdot J_2^4\cdot J_4\cdot J_8 +\\
35303694973045127054884864/1215\cdot J_2^4\cdot J_5\cdot J_7 -
14140368315298612206632960/6561\cdot J_2^4\cdot J_6^2 +\\
1018515427675205805688225792/164025\cdot J_2^3\cdot J_3^3\cdot J_5 -
336610288267269639778598912/54675\cdot J_2^3\cdot J_3^2\cdot J_4^2 -\\
74969432406206948684333056/2187\cdot J_2^3\cdot J_3^2\cdot J_8 +
6649457326591889981308928/45\cdot J_2^3\cdot J_3\cdot J_4\cdot J_7 -\\
619030254390189753846726656/10935\cdot J_2^3\cdot J_3\cdot J_5\cdot J_6 -
3608451735885372384297877504/10935\cdot J_2^3\cdot J_4^2\cdot J_6 +\\
34814250412419307827888128/405\cdot J_2^3\cdot J_4\cdot J_5^2 -
925614958438175247473573888/1215\cdot J_2^3\cdot J_4\cdot J_{10} +\\
4710749291369895452213248/135\cdot J_2^3\cdot J_5\cdot J_9 -
494819283286030519883530240/729\cdot J_2^3\cdot J_6\cdot J_8 +\\
14537770989836507807744\cdot J_2^3\cdot J_7^2 - 431154367844645937545740288/98415\cdot J_2^2\cdot
J_3^4\cdot J_4\\ - 19613507262065954888089600/2187\cdot J_2^2\cdot J_3^3\cdot J_7 -
146707873700134638489436160/729\cdot J_2^2\cdot J_3^2\cdot J_4\cdot J_6 +\\
48180602381764409738395648/6075\cdot J_2^2\cdot J_3^2\cdot J_5^2 +
40705853229537397508669440/243\cdot J_2^2\cdot J_3^2\cdot J_{10} -\\
1387066956958698274029568/1215\cdot J_2^2\cdot J_3\cdot J_4^2\cdot J_5 -
7435068187959987393265664/81\cdot J_2^2\cdot J_3\cdot J_4\cdot J_9 +\\
36311983381615293444915200/81\cdot J_2^2\cdot J_3\cdot J_5\cdot J_8 +
347874930144850910025613312/729\cdot J_2^2\cdot J_3\cdot J_6\cdot J_7 +\\
1255008409772073192393801728/6075\cdot J_2^2\cdot J_4^4 +
86331981484958040096505856/27\cdot J_2^2\cdot J_4^2\cdot J_8 -\\
46372544865756833808121856/135\cdot J_2^2\cdot J_4\cdot J_5\cdot J_7 -
4872843022291903687675609088/10935\cdot J_2^2\cdot J_4\cdot J_6^2 +\\
273116722157213671548256256/405\cdot J_2^2\cdot J_5^2\cdot J_6 -
639504861274391454026825728/81\cdot J_2^2\cdot J_6\cdot J_{10} +\\
7611332353115259569963008/9\cdot J_2^2\cdot J_7\cdot J_9 -
45986422419474003848593408/32805\cdot J_2\cdot J_3^6 -\\
1362337960585246054864125952/19683\cdot J_2\cdot J_3^4\cdot J_6 -
31691914062729600254869504/1215\cdot J_2\cdot J_3^3\cdot J_4\cdot J_5 +\\
6356506589402467076669440/243\cdot J_2\cdot J_3^3\cdot J_9 +
566267050476982096994762752/6075\cdot J_2\cdot J_3^2\cdot J_4^3 +
\\
43732296667447222171860992/81\cdot J_2\cdot J_3^2\cdot J_4\cdot J_8 +
3904625658485018957185024/27\cdot J_2\cdot J_3^2\cdot J_5\cdot J_7 -
\\
1286370505156353475162931200/6561\cdot J_2\cdot J_3^2\cdot J_6^2 -
5208687002242783664668672/3\cdot J_2\cdot J_3\cdot J_4^2\cdot J_7 +\\
14323812823184402857066496/27\cdot J_2\cdot J_3\cdot J_4\cdot J_5\cdot J_6 -
23793150319901049836535808/225\cdot J_2\cdot J_3\cdot J_5^3 -
\end{multline*}

\begin{multline*}
58820106293726805219082240/27\cdot J_2\cdot J_3\cdot J_5\cdot J_{10} +
446780612957806169881051136/81\cdot J_2\cdot J_3\cdot J_6\cdot J_9 +\\
21094187031790256033628160/27\cdot J_2\cdot J_3\cdot J_7\cdot J_8 +
2810488743488952870562168832/6075\cdot J_2\cdot J_4^3\cdot J_6 -
\\
40612648665960118006841344/75\cdot J_2\cdot J_4^2\cdot J_5^2 +
1899511296268009901110853632/135\cdot J_2\cdot J_4^2\cdot J_{10} -\\
219267176513018334369808384/45\cdot J_2\cdot J_4\cdot J_5\cdot J_9 -
1396117748045736738251866112/81\cdot J_2\cdot J_4\cdot J_6\cdot J_8 +\\
2505606645422999207936000/9\cdot J_2\cdot J_4\cdot J_7^2 -
11479503724865816564334592/3\cdot J_2\cdot J_5^2\cdot J_8 +\\
118276542650689432441585664/27\cdot J_2\cdot J_5\cdot J_6\cdot J_7 -
5919291988926361071008088064/2187\cdot J_2\cdot J_6^3 +\\
22888858956213483842043904/3\cdot J_2\cdot J_9^2 -
7426491921950305668825088/1215\cdot J_3^5\cdot J_5 +\\
43658225372385848292540416/10935\cdot J_3^4\cdot J_4^2 -
3676086792284031502254080/81\cdot J_3^4\cdot J_8 +
\\
13964097972615215174385664/243\cdot J_3^3\cdot J_4\cdot J_7 +
4729944765365651511443456/27\cdot J_3^3\cdot J_5\cdot J_6 -\\
326349873401360496115843072/729\cdot J_3^2\cdot J_4^2\cdot J_6 +
8284834323872744942338048/135\cdot J_3^2\cdot J_4\cdot J_5^2 +\\
38610817879262287831760896/81\cdot J_3^2\cdot J_4\cdot J_{10} +
36142134430859893179154432/27\cdot J_3^2\cdot J_5\cdot J_9 -\\
319373012455443877561630720/243\cdot J_3^2\cdot J_6\cdot J_8 +
5094152176839330837299200/27\cdot J_3^2\cdot J_7^2 -
\\
202070939385898091820351488/675\cdot J_3\cdot J_4^3\cdot J_5 -
263370259894608708776230912/135\cdot J_3\cdot J_4^2\cdot J_9 -\\
45560138686768979218792448/9\cdot J_3\cdot J_4\cdot J_5\cdot J_8 +
176182578405149234752913408/81\cdot J_3\cdot J_4\cdot J_6\cdot J_7 +\\
6425958693517651685146624/3\cdot J_3\cdot J_5^2\cdot J_7 -
5585246833065391336456192/9\cdot J_3\cdot J_5\cdot J_6^2 -\\
8522447483696509941186560/9\cdot J_3\cdot J_7\cdot J_{10} + 143846738599820378507313152/625\cdot J_4^5\\
+ 450198288326268137643180032/45\cdot J_4^3\cdot J_8 -
48017713874938717785817088/15\cdot J_4^2\cdot J_5\cdot J_7 -\\
179304344085561207987109888/405\cdot J_4^2\cdot J_6^2 +
35010518636149680570892288/45\cdot J_4\cdot J_5^2\cdot J_6 -\\
510546782523321128191000576/27\cdot J_4\cdot J_6\cdot J_{10} +
9321426935293057748172800/3\cdot J_4\cdot J_7\cdot J_9\\
 + 20553746392321192334721024/25\cdot J_5^4 
-
8309386296604097192656896\cdot J_5^2\cdot J_{10} +\\ 3200265993877873120772096/3\cdot J_5\cdot J_6\cdot J_9 -
2979527538245459530219520\cdot J_5\cdot J_7\cdot J_8 + \\176058111538199840570736640/81\cdot J_6^2\cdot J_8
+ 16305703093807098101760\cdot J_6\cdot J_7^2;
\end{multline*}

\bibliographystyle{alphaabbr}
\bibliography{sample} 

%

\end{document}

%% file: authorinfo.tex

\author[Lorenzo]{Elisa Lorenzo Garc\'ia}
\address{%
	Elisa Lorenzo Garc\'ia
  Univ Rennes, CNRS, IRMAR - UMR 6625, F-35000
 Rennes, %
  France. %
}
\email{elisa.lorenzogarcia@univ-rennes1.fr}


